\documentclass[nopreprintline,12pt]{elsarticle}

 \usepackage{amsthm}
\usepackage{lineno}
\usepackage{latexsym,amscd,amsfonts,enumerate,supertabular}
\usepackage{tabularx}
\usepackage{array}
\usepackage{bigstrut}
\usepackage{graphicx}
\usepackage{epsfig}
\usepackage{amssymb}
\usepackage[intlimits]{amsmath}
\usepackage{mathrsfs} 
\usepackage{algorithm}

\usepackage{multirow}
\usepackage{hhline}	
\usepackage{makecell}
\usepackage{color}
\usepackage{xcolor}
\usepackage{float}
\usepackage{graphicx,subfigure}
\usepackage{mathabx}

\definecolor{darkmagenta}{rgb}{0.55, 0.0, 0.55}

\newcolumntype{R}[1]{>{\raggedleft\let\newline\\\arraybackslash\hspace{0pt}}m{#1}}
\newcolumntype{L}[1]{>{\raggedright\let\newline\\\arraybackslash\hspace{0pt}}m{#1}}
\newcolumntype{C}[1]{>{\centering\let\newline\\\arraybackslash\hspace{0pt}}m{#1}}

\numberwithin{equation}{section}
\newtheorem{lemma}[]{Lemma}
\numberwithin{lemma}{section}
\newtheorem{theorem}[]{Theorem}
\numberwithin{theorem}{section}
\numberwithin{algorithm}{section}
\newtheorem{remark}[]{Remark}
\numberwithin{remark}{section}
\newtheorem{example}[]{Example}
\numberwithin{example}{section}
\newtheorem{definition}[]{Definition}
\numberwithin{definition}{section}
\newtheorem{corallary}[]{Corallary}
\numberwithin{corallary}{section}

\newcommand{\N}{\mathbb N}
\newcommand{\R}{\mathbb R}
\newcommand{\XX}{\mathfrak X}

\newcommand{\A}{\mathscr A}

\newcommand{\cl}{\text{cl}}

\begin{document}
\begin{frontmatter}
\title{{\bf
A representation theorem for set-valued submartingales
}}
\author[ioit]{Luc T. Tuyen\corref{cor1}}
 \ead{tuyenlt@ioit.ac.vn}
\author[msu]{Vu T. Luan}
\ead{luan@math.msstate.edu}

\cortext[cor1]{Corresponding author}
\address[ioit]{Institute of Information Technology, Vietnam Academy of Science and Technology, Hanoi, Vietnam}
\address[msu]{Department of Mathematics and Statistics, Mississippi State University,
 410 Allen Hall, 175 President's Circle, Mississippi State, MS, 39762, USA}
\begin{abstract}
\small  

The integral representation theorem for martingales has been widely used in probability theory.
In this work, we propose and prove a general representation theorem for a class of set-valued submartingales. 
 We also extend the stochastic integral representation for non-trivial initial set-valued martingales.
Moreover, we show that this result covers the existing ones in the literature for both degenerated and non-degenerated set-valued martingales.

\small  
\end{abstract}

\begin{keyword}
\small 
Set-valued integrals
\sep
Set-valued martingales
\sep 
Random sets
\sep 
Set-valued stochastic processes
\sep
 Interval-Valued martingale
 \sep
 Martingale representation theorem
\end{keyword}
\end{frontmatter}


\section{Introduction}\label{intro}
Martingales 
(i.e., the single-valued case) have played an important role in many statistical applications \cite{hall2014martingale}. In the random sets (or set-valued)
framework, the definitions  of set-valued martingales, sub-martingales and super-martingales have been introduced by Hiai and Umegaki \cite{hiai1977integrals}. In recent years, the limit theorems of 
such type of martingales in the discrete setting have been extended and represented in \cite{molchanov2017theory} and \cite{li2013limit}. 
While the applications of stochastic processes in general and of martingales in particular have been widely used long time ago in many statistical fields, applications of set-valued stochastic only received more attention in recent years, see e.g., in mathematical economics \cite{arutyunov2016convex}, in control theory \cite{aubin2000introduction} and finance \cite{Kisielewicz2020set} and set-valued optimization problems \cite{peng2022painleve, peng2022Tech}.

To incorporate  the set-valued variable theory into practical applications, the integration for set-valued functions has been extended, see e.g.,  Aumann \cite{aumann1965integrals}, Hukuhara \cite{hukuhara1967integration}, Debreu \cite{debreu1967integration}, Henstock \cite{wu2000henstock} and Bartle \cite{wu2001set}. 
For set-valued stochastic processes, we refer the readers to several works of
Kisielewicz \cite{kisielewicz1997set,kisielewicz2012some,kisielewicz2014martingale} and Kim \cite{kim1997integrals,kim1999stochastic,jung2003set}. 
Zhang et al. \cite{zhang2009set,zhang2020submartingale} has obtained some martingale and sub-martingale properties of set-valued stochastic integral with respect to Brownian motion  as well as to the compensated Poisson measure. Specifically, the authors have showed in an M-type 2 Banach space that the set-valued stochastic integral becomes a set-valued submartingle and the Castaing representation for the set-valued stochastic integral is also obtained.
It is shown in  \cite{kisielewicz2014martingale} that for every set-valued martingale $F=(F_t)_{t\ge 0}$ with $F_0=\{0\}$, there exists a integral representation of $F_t$ in the context of the generalized set-valued stochastic integrals. 
We note, however, that for any set-valued martingale $F=(F_t)_{t\ge 0}$ with $F_0=\{0\}$ one has $E[F_t|\mathcal F_0]=F_0=\{0\}$ for all $t\ge 0$, where $E[\cdot|\cdot]$ denotes the conditional expectation of set-valued random variables (in the Aumann's sense). Therefore, one can see that  $(F_t)_{t\ge 0}$ degenerates to the singleton one (see \cite[Theorem 2.1.72]{molchanov2017theory}). This means that the representation theorem for set-valued martingales proposed in \cite{kisielewicz2014martingale} holds for single-valued martingales only. From this point of view, in some real world applications it is more reasonable to consider the condition $E[F_t|\mathcal F_0]\ni \{0\}$ for all $t\ge 0$ instead of the equal one, see \cite{tuyen2022weak}).
The authors of \cite{zhang2020remarks} then proposed a representation theorem for non-degenerated set-valued martingales that holds for a special case, $\displaystyle F_t=C+\left\{\int_0^t g_s dB_s\right\}$, where $C$ is a bounded, closed and convex subset of a Banach space $\mathfrak X$ and $(g_t)_{t\in [0,T]}$ is an $\mathfrak X$-valued stochastic process.

The aim of this work is to fill this gap by presenting a general result, showing that under some reasonable assumptions on the set-valued submartingale $\{F_t\}_{t\ge 0}$  the representation $F_t=\displaystyle \left\{\int_0^t \mathcal G_s dB_s\right\}_{t\ge 0}$ holds not only for the degenerated set-valued martingales but also for the non-degenerated one.

The paper is organized as follows. In Section \ref{preliminaries}, we briefly introduce the notations used throughout the paper and comment on some previous related results of set-valued stochastic processes. Section \ref{mainresults} devotes to the our main result, where we state and prove some general representation theorems for both non-singleton initial set-valued $\mathbb{F}$-martingales and submartingales.

\section{Notations and Preliminaries}\label{preliminaries}
Throughout of this paper, we consider a complete filtered probability space $\mathcal P_{\mathbb{F}}=(\Omega,\mathcal F,\mathbb{F},P)$ equipped with a filtration $\mathbb{F}=(\mathcal F_t)_{t\ge 0}$, where the usual conditions are satisfied. We denote by $(\mathfrak{X},\|\cdot\|)$ the separable Banach space with the Borel sigma-algebra $\mathcal B(\mathfrak{X})$, and by $K(\XX)$ the set of  all nonempty closed subsets of $\XX$. By adding the suffixes $c$, $bc$, $kc$ to $K(\XX)$, we 
denote $K_c(\XX)$, $K_{bc}(\XX)$, and $K_{kc}(\XX)$  as  the family of all nonempty closed convex, nonempty bounded closed convex, nonempty compact convex subsets of $\XX$, respectively. Denote $\|C\|=\sup\{\|x\|:x\in C\}$ for all $C\in K(\XX)$. We denote by $L^1[\Omega;\XX]$ (if $\XX=\R$ then it is written by $L^1$ for short) the set of all Bochner integrable random variables taking their values on $\XX$. For $1\le p< \infty$, denote $L^{p}[\Omega,\mathcal F,P;\XX]=L^p[\Omega;\XX]$ be the Banach space of measurable functions $f:\Omega\longrightarrow \XX$ with the norm
$$ \left\|f\right\|_p=\left\{\int_{\Omega}\left\|f(\omega)\right\|^p dP\right\}^{1/p},\, 1\le p<\infty. $$
Denote $L^p[\Omega,\mathcal F,P;\R]=L^p$ be the usual Banach space of real-valued functions with finite $p$-th moments.

Recall that a set-valued mapping $F:\Omega\to \mathbf{K}(\XX)$ is called a set-valued random variable or a random set if for each closed subset $C$ of $\XX$, $\{\omega \in\Omega: F(\omega)\cap C\ne \emptyset\}\in\mathcal{F}$. 	A function taking its values in $\XX$, $f:\Omega\to \XX$ is called a \emph{selection} of a set-valued mapping $F:\Omega\to \mathbf{K}(\XX)$ if $f(\omega)\in F(\omega)$ for all $\omega\in\Omega$. A function $f$ is called an \emph{almost everywhere selection} of $F$ if $f(\omega)\in F(\omega)$ for almost everywhere $\omega\in\Omega$. Denote by $U[\Omega,\mathcal{F},P;\mathbf{K}(\XX)]$ the family of all set-valued random variables, simply written as
 $U[\Omega;\mathbf{K}(\XX)]$.\\
For $1\le p\le \infty$, denote by
$$ S_F^p(\mathcal F)=\left\{f\in L^p[\Omega;\XX]: f(\omega)\in F(\omega)\, , a.e.\right\} $$ 
the set of $p$-order integrable selections of the set-valued random variable $F$. For simplicity, the symbol $S_F$ stands for $S^1_F$. Notice that $S_F^p$ is a closed subset of $L^p[\Omega;\XX]$.

Let $\mathcal{M}$ be a set of measurable functions $f:\Omega \to \XX$. $\mathcal{M}$ is called decomposable (with respect to $\mathcal F$) if for all $f_1, f_2\in\mathcal{M}$ and $A\in\mathcal F$ then $I_Af_1+I_{\Omega\setminus A}f_2\in\mathcal{M}$.

\begin{theorem}(\cite{li2013limit})\label{theo:closed_set_selection}
	Let $\mathcal{M}$ be a non-empty closed subset of $L^p[\Omega;\XX]$ and $1\le p <\infty$. Then there  exists an $F\in U[\Omega; \bf{K}(\XX)]$ such that $\mathcal{M}=S^p_F$, if and only if $\mathcal{M}$ is decomposable. 
\end{theorem}

A set-valued random variable $F$ is called integrable if $S^1_F$ is non-empty. $F$ is called $L^p$-integrably bounded (or strongly integrable) if there exists $\rho\in L^p[\Omega,\mathcal F,\R]$ such that $\|x\|\le \rho(\omega)$ for all $x\in F(\omega)$ and $\omega\in\Omega$. Let
 $L^p[\Omega,\mathcal{F},P;\mathbf{K}(\XX)]$ denotes the family of all $L^p$-integrably bounded set-valued random variables (we simply write it as $L^p[\Omega;\mathbf{K}(\XX)]$).
  If $\rho\in L^2[\Omega,\mathcal F,\R]$ then $F$ is called square integrably bound.\\
For a set-valued random variable $F\in L^p[\Omega;\mathbf{K}(\XX)]$, $F$ can be represented by its $p$-order integrable selections. It is called Castaing representation of $F$. 

\begin{theorem}(\cite{hiai1977integrals})\label{Castaing_Rep}
	For a set-valued random variable $F\in U[\Omega;\mathbf{K}(\XX)]$, there exists a sequence $\{f^i:i\in\N\}\subset S^p_F$ such that $$F(\omega)=\text{cl}\{f^i(\omega):i\in\N\}$$ for all $\omega\in\Omega$, where the closure is taken in $\XX$. \\
	Moreover, denote by $\overline{\text{dec}}$ the decomposable closure of the sequence $\{f^i:i\in\N\}$ in $L^p[\Omega;\XX]$, we have
	$$ S^p_F(\mathcal F)=\overline{\text{dec}}\{f^i:i=1,2,...\}.$$
\end{theorem}

For a set-valued random variable $F$, the Aumann integral of $F$ is defined by 
$$ \int_{\Omega}F dP=\left\{\int_{\Omega}f dP:\, f\in S_F\right\}, $$
where $\displaystyle\int_{\Omega}fdP$ is usual integral Bochner in $L^1[\Omega;\XX]$. Because $\displaystyle \int_{\Omega}F dP$ is not closed in general except under some conditions such as $\XX$ has Radon Nikodym property and $F\in L^1[\Omega;K_{kc}(\XX)]$ or $\XX$ is reflexive and $F\in L^1[\Omega;K_{c}(\XX)]$ (notice that $L^p[\Omega;\XX]$ with $1<p<\infty$ and any finite dimensional spaces are reflexive), the expectation of $F$ is given by $$ E[F]=\cl\int_{\Omega}F dP $$ with the closure is taken in $\XX$.

Let $F\in U[\Omega;\mathbf{K}(\XX)]$ with $S_F\ne \emptyset$. Then there exists a unique $\A$-measurable element of $U[\Omega,\A,P;\textbf{K}(\XX)]$, denoted by $E[F|\A]$ \cite{li2013limit}, such that
\begin{equation}\label{eq:kyvongdk}
	S_{E[F|\A]}(\A) = \cl\left\{E[f|\A]:\, f\in S_F\right\},
\end{equation}
where the closure are taken in $L^1[\Omega;\XX]$. The set-valued random variable $E[F|\A]$ satisfying \eqref{eq:kyvongdk} is called the conditional expectation of $F$ with respect to $\A$. 
Clearly, if $\XX$ is reflexive and $F$ is weakly compact convex then $\left\{E[f|\A]:\, f\in S_F\right\}$ is closed in $L^1[\Omega;\XX]$, see \cite{assani1982parties}.

A set-valued random process $F=(F_t)_{t\ge 0}$ is said to be an $\mathbb{F}$-adapted convex set-valued process on $\mathcal P_{\mathbb{F}}$ if $F_t\in U[\Omega,\mathcal F_t,P;K_{c}(\XX)]$ for each $t\ge 0$. $F$ is said to be set-valued $\mathbb{F}$-martingale if $S_{F_t}({\mathcal F_t})\ne \emptyset$ and $E[F_t|\mathcal F_s]=F_s\, a.s.$ for every $0\le s\le t <\infty$. $F$ is called a set-valued $\mathbb{F}$-submartingale (resp. supermartingale) if for any $0\le s\le t$, $E[F_t|\mathcal F_s]\supset F_s$ (resp. $E[F_t|\mathcal F_s]\subset F_s$). The expectation in this definition is given as in \eqref{eq:kyvongdk}.

Given $\XX$ is a separable M-type 2 Banach space. Let $\mathcal L^2_{\mathbb F}:=L^2[\R^+\times \Omega,\beta[0,t]\times\mathcal F,P;\XX]$ be the family of all $L^2$-measurable stochastic process in $\XX$, we denote by $\mathcal P(\mathcal L^2_{\mathbb F})$ the all nonempty subsets of the space $\mathcal L^2_{\mathbb F}$. For fixed $t\ge 0$, denote by $J_B^t$ the mapping by setting $J_B^0=0$ and $\displaystyle J_B^t(\varphi)=\int_0^t\varphi_{\tau}dB_{\tau}$ for $t>0$ and every $\varphi=(\varphi_t)_{t\ge 0}\in \mathcal L^2_{\mathbb F}$ (the integral is defined in \cite{zhang2009set}). Let $\mathcal G \in \mathcal P(\mathcal L^2_{\mathbb F})$, for fixed $t >0$ there exists only one $\mathcal F_t$-measurable set-valued random variable $F_t:\Omega \to \mathbf K(\XX)$ such that $S_{F_t}(\mathcal F_t)=\overline{\text{dec}}J_B^t(\mathcal G)$ (see Theorem \ref{theo:closed_set_selection}). $F_t$, is now denoted by $\displaystyle \int_0^t\mathcal G dB_{\tau}$ and called the generalized set-valued stochastic integral of $\mathcal G$ on $[0,t]$ w.r.t the Brownian motion $B$. A set-valued martingale representation theorem has been established as follows.

\begin{theorem} (\cite{kisielewicz2014martingale})\label{theo:kisielewicz2014}
	For every set-valued $\mathbb F$-martingale $F=(F_t)_{t\ge 0}$ defined on a filtered probability space $(\Omega,\mathcal F,\mathbb F,P)$ with the completed natural filtration $\mathbb F=(\mathcal F_t)_{t\ge 0}$ of an $m$-dimensional Brownian motion $B=(B_t)_{t\ge 0}$ defined on $(\Omega,\mathcal F,P)$ and such that $F_0=\{0\}$,  there exists a set $\mathcal G\in \mathcal P(\mathcal L^2_{\mathbb F})$ such that $\displaystyle F_t=\int_0^t\mathcal GdB_{\tau}$ a.s. for every $t\ge 0$.
\end{theorem}

\begin{remark}
One can show that Theorem~\ref{theo:kisielewicz2014} holds only when $F$ is singleton. First,
we note that, for a set-valued martingale $F$, one has $E[F_t]=E[F_0]$ for all $t\ge 0$. Therefore, under the assumption  $F_0=\{0\}$ of Theorem~\ref{theo:kisielewicz2014}, $E[F_t]=\{0\}$. Now, using \cite[Theorem~2.1.72]{molchanov2017theory}, one derives that $F_t$ is degenerated to singleton. 
Conversely, for a set $\mathcal G\in \mathcal P(\mathcal L^2_{\mathbb F})$, we see that  $\displaystyle F_t=\int_0^t\mathcal GdB_{\tau}$ is not a non-degenerated set-valued $\mathbb F$-martingale but is a set-valued submartingale (see \cite{zhang2009set,zhang2020submartingale}).

\end{remark}

An extended version of Theorem~\ref{theo:kisielewicz2014}  for  the non-degenerated case  is thus proposed in \cite{zhang2020remarks}, which is recalled here as follows.

\begin{theorem}(\cite{zhang2020remarks})\label{theo:zhang2020remarks}
	Let $(\XX,\|.\|)$ be a separable reflexive M-type 2 Banach space and $\{M_t, 0\le t\le T\}$ be a $\mathbf{K}_c(\XX)$-valued $\mathbb F$-martingale. Then, the following statements are equivalent:\\
	(i) There exists a set-valued stochastic process $(G_t)_{0\le t\le T}$ such that
	$$ M_t=E(M_0)+\int_0^t G_sdB_s \, \ a.s.\, \text{for each }t.$$
	(ii) There exists an $\XX$-valued stochastic process $g=\{g_t:t\in [0,T]\}$ and a bounded, closed and convex subset $C\subset \XX$ such that
	$$M_t=C+\left\{\int_0^t g_s dB_s\right\}\, a.s.\, \text{for each }t. $$
	(iii) There exists a sequence $\{f^1,\cdots,f^n,\cdots\}$ of $\XX$-valued $\mathbb F$-martingales such that
	$$M_t=\text{cl}\{f^1,\cdots,f^n,\cdots \} \, a.s.\, \text{for each }t$$ and $f^i_t-f^j_t$ is non-random and independent of $t$ for any $i,j\ge 1$.\\
Here, the definition of M-type 2 Banach space and the set-valued stochastic integral w.r.t Brownian motion $\displaystyle I_t(G)=\int_0^tG_sdB_s$  are given in \cite{zhang2009set}.
\end{theorem}

\begin{remark}
It is clear that this result  still holds only in a very special case, i.e., $M_t$ equals to sum of a set $C$ and a singleton. The representation (i), $\displaystyle M_t=E(M_0)+\int_0^t G_sdB_s \, \ a.s.\, \text{for each }t$, implies that $\{M_t\}$ is a set-valued submartingale since $\displaystyle \left\{\int_0^t G_sdB_s\right\}$ is a set-valued submartingale. The statements (ii) and (iii) imply that the special case of submartingale (i.e. martingale) holds only when $\{G_t\}$ degenerates to a singleton. 
\end{remark}

\section{Main results}\label{mainresults}
In this section, we present a new representation theorem for a trivial initial set-valued $\mathbb F$-submartingale with natural filtration $\mathbb F=(\mathcal F_t)_{t\ge 0}$ of an $m$-dimensional Brownian motion $B=(B_t)_{t\ge 0}$ defined on $(\Omega,\mathcal F, P)$. Finally, we propose an extension of integral representation for non-trivial initial set-valued martingale.

\subsection{Integral representation for trivial initial set-valued submartingales}
First, we have the following lemma.
\begin{lemma}\label{lem:submartingale_Castaning}
Let $F=(F_t)_{t\ge 0}$ be a set-valued $\mathcal F$-submartingale in $L^p[\Omega,\mathcal F,\mathbb F,P;$ $ \mathbf{K}_{bc}(\XX)]$ ($p\ge 1$) with $\XX$ is reflexive. Then, there exists a system of $\XX$-valued random variables $\{f_{r,t}^k:t\ge 0,r\ge 0,k\in\N\}$ such that 
	\begin{itemize}
		\item[(i)] For each $k\in\N$, $r\ge 0$, the sequence $(f_{r,t}^k)_{t\ge 0}$ is a $\XX$- valued $\mathbb{F}$-martingale.
		\item[(ii)] For each $k\in\N$, $r,t\ge 0$ with $0\le r\le t$, $f_{r,t}^k\in S_{F_t}^p(\mathcal F_t)$.
		\item[(iii)] For each $t\ge 0$, $F_t(\omega)=cl_{\XX}\{f_{r,t}^k(\omega):0\le r\le t,k\in\N\}$ for all $\omega\in\Omega$. In addition,
		$$ S_{F_t}(\mathcal F_t)=\overline{\text{dec}}\{f_{r,t}^k(\omega):0\le r\le t,k\in\N\},$$
	\end{itemize}
	where $\text{cl}_{\XX}$ denotes the closure  taken in $\XX$.
\end{lemma}
\begin{proof}
	For fixed $r\ge 0$, since $F_{r}\in L^{p}\left[ \Omega ,\mathcal{F},\mathbb{F},P:K_{bc}\left( \mathfrak{X} \right) \right]$  and by Castaing representation (Theorem \ref{Castaing_Rep}), there is a countable family $\left\{ f_{r,r}^{k}\in {S}^p_{F_r}(\mathcal F_r),k\in \mathbb{N} \right\}$ such that $F_{r}(\omega)=\text{cl}_{\XX}\left\{ f_{r,r}^{k}(\omega ):k\in \mathbb{N} \right\}$ for all $\omega \in \Omega $. For each $k\in \mathbb{N}$ and $r\ge 0$, we define a $\mathfrak X$-valued martingale $\left(f_{r,t}^k\right)_{t\ge 0}$ as follows. For $t\le r$, define $f_{r,t}^k= E[f_{r,r}|\mathcal F_t]$. 
For $t>r$, since $\mathfrak{X}$ is reflexive and due to Eberlein-Smulyan theorem, $F_t$ is weakly compact for all $t\ge 0$ and $\omega\in\Omega$. Thus, 
	$$ S_{E[F_{t}|{\mathcal{F}}_{r}]}\left( \mathcal{F}_r \right)=\left\{ E[f|\mathcal{F}_{r}]:f\in S_{F_t}\left(\mathcal{F}_{t}\right) \right\} $$
	 is closed in $L^p[\Omega ,{\XX}]$ by \cite{assani1982parties}. 
Since $\left( F_{t} \right)_{t\ge 0}$ is a set-valued submartingale, we have 
	 $$ S_{F_{r}}\left(\mathcal{F}_{r} \right)\subset \left\{ E[f|\mathcal{F}_{r}]:f\in S_{F_t}\left(\mathcal{F}_{t}\right) \right\} $$ 
	 for every $0\le r\le t<\infty $.
	Thus, for $f_{r,r}^{k}\in S_{F_r}\left(\mathcal{F}_{r}\right)$ there exists $f_{r,t}^{k}\in S_{F_t}\left(\mathcal{F}_{t}\right)$ such that $E\left[ f_{r,t}^{k}|\mathcal{F}_{r}\right]={{f}_{r,r}}$. Hence, $f=\left( f_{r,t}^{k} \right)_{t\ge 0}$ satisfies the condition (i) and (ii). 
	
	On the other hand, for each $t\ge 0$ we have 
	$$F_{t}(\omega )=\text{cl}_{\XX}\left\{ f_{t,t}^{k}(\omega ):k\in \mathbb{N} \right\} $$ 
for all $\omega \in \Omega $. Since  $f_{r,t}^{k}\in S_{F_t}\left(\mathcal{F}_{t}\right)$ for all $0\le r\le t$ and $k\in \mathbb{N}$, we obtain $F_{t}(\omega) =\text{cl}_{\mathfrak{X}}\left\{ f_{r,t}^{k}(\omega ):0\le r\le t,k\in \mathbb{N} \right\}$ for every $\omega \in \Omega $. 
Next, it is clear that $\overline{\text{dec}}\left\{ f_{r,t}^{k}:0\le r\le t,k\in \mathbb{N} \right\}\subset S_{F_t}(\mathcal F_t)$. Conversely, for every $f\in S_{F_t}(\mathcal F_t)$ and every $\epsilon>0$, the result of F. Hiai and H. Umegaki  (\cite{hiai1977integrals}, Lemma 1.3) shows that there exists a finite measurable partition $\left\{ A_1, A_2, \cdots, A_n \right\}$ of $\Omega$ and a family $\left( f_{r,t}^{k_i} \right)_{i=1}^{n}\subset \left\{ f_{r,t}^{k}:k\ge 1 \right\}$ such that $E\left\| f-\sum\limits_{i=1}^{n}{1_{A_{i}}f_{r,t}^{k_i}} \right\|_{\mathfrak{X}}\le \epsilon $, 
which implies that $f\in \overline{\text{dec}}\left\{ f_{r,t}^{k}:0\le r\le t,k\in \mathbb{N} \right\}$.  
Thus,  $S_{F_t}\left( \mathcal{F}_{t}\right)=\overline{\text{dec}}\left\{ f_{r,t}^{k}:0\le r\le t,k\in \mathbb{N} \right\}$,
which concludes (iii).
\end{proof}

With these preparations at hand, we are now ready to present the main result of this paper.

\begin{theorem}[Set-valued submartingale representation theorem]\label{theo:submart_representation} 
	Given $\XX$ is a separable reflexive M-type 2 Banach space. Let $F=(F_t)_{t\ge 0}$ with\\ $F_t\in L^2[\Omega,\mathcal F,\mathbb F,P;\mathbf{K}_{bc}(\XX)]$ be a bounded closed convex set-valued $\mathbb F$-submartingale defined on a filtered probability space $\mathcal{P}_{\mathbb{F}}$ with completed natural filtration $\mathbb{F}=\left(\mathcal{F}_{t} \right)_{t\ge 0}$ of an $m$-dimensional  Brownian motion $B=(B_t)_{t\ge 0}$ defined on $\left( \Omega ,\mathcal{F},P \right)$.
Suppose that ${F}_{0}=\left\{ 0 \right\}$ a.s. and for each $t\ge 0$  the Castaining  representation ${F}_{t}=\text{cl}\left\{ f_{t}^{k}:k\in \mathbb{N} \right\}$ satisfies $E\left[ f_{t}^{k}|{\mathcal{F}_{s}} \right]\in S_{F_s}\left(\mathcal{F}_{s}\right)$  for all $k\in \mathbb{N}$ and $s<t$, then there exists a set $\mathcal{G}\subset \mathcal{P}\left( \mathcal{L}_{\mathbb{F}}^{2} \right)$ such that $\displaystyle F_{t}=\int_{0}^{t}\mathcal{G}d{B_{\tau }}$, where the integral is in the sense of generalized set-valued stochastic integral.
\end{theorem}
\begin{proof}
	By using Lemma \ref{lem:submartingale_Castaning}, we have 
	$$S_{F_t}\left( {\mathcal{F}_{t}} \right)=\overline{\text{dec}}\left\{ f_{r,t}^{k}:0\le r\le t,k\in \mathbb{N} \right\},$$  
	where $\left( f_{r,t}^{k} \right)_{t\ge 0}$ is a $\mathbb F$-martingale for each $r\ge 0$ and $k\in \mathbb{N}$. 
	
Since  $E\left[ f_{t}^{k}|\mathcal{F}_{s}\right]\in S_{F_s}\left( {\mathcal{F}_{s}} \right)$ for all $s<t$ and ${F}_{0}=\left\{ 0 \right\}$, for each $r\ge 0$ we choose $f_{r,t}^k\in S_{F_t}(\mathcal F_t)$ such that $E[f_{r,r}^k|\mathcal F_t]=f_{r,t}^k$ for all $t<r$ and $k\in \mathbb N$. Consequently, for each $t\ge 0$ we have
	\begin{equation}\label{eq:1}
		S_{F_t}\left( {\mathcal{F}_{t}} \right)=\overline{\text{dec}}\left\{ f_{r,t}^{k}:r\ge 0,k\in \mathbb{N} \right\},
	\end{equation}
	where $\left( f_{r,t}^{k} \right)_{t\ge 0}$ is a $\mathbb F$-martingale with $f_{r,0}^k=0$ for each $r\ge 0$ and $k\in \mathbb{N}$.
	
As a result of the traditional martingale representation theorem  for every martingale selector $f_{r}^{k}=\left( f_{r,t}^{k} \right)_{t\ge 0}$, there exists exactly one (with respect to the equality a.e.) $g^{f_r^k}\in \mathcal{L}_{\mathbb{F}}^{2}$ such that $\displaystyle f_{r,t}^{k}=\int _{0}^{t}g_{\tau }^{f_r^k}d{B_{\tau }}$. For $t\ge 0$, \\
	let $\mathcal{G}=\left\{ g= \left(g_{t}^{f_r^k} \right)_{t\ge 0}\in \mathcal{L}_{\mathbb{F}}^{2}:r\ge 0,k\in \mathbb{N} \right\}$. By the definition of the generalized set-valued stochastic integrals, we have
	\begin{equation}\label{eq:2}
		S_{\int_{0}^{t}{\mathcal{G}}d{B_s}}(\mathcal F_t)=\overline{\text{dec}}\left\{\int _{0}^{t}g_{s}^{f_r^k}d{B_s}:g_{s}^{f_r^k}\in \mathcal{G} \right\}.
	\end{equation}
	
	From \eqref{eq:1} and \eqref{eq:2} we have $${{F}_{t}}=\int_{0}^{t}{\mathcal{G}}d{{B}_{s}}.$$
\end{proof}
\begin{remark}
	We have the following observations
	\begin{itemize}
		\item[1.] Theorem \ref{theo:submart_representation} covers the result of Theorem \ref{theo:kisielewicz2014} because with a zero origin set-valued martingale $\{F_t\}$, it degenerates to a singleton, $F_t=\{f_t\}$. The Castaing representation of $F_t$ now becomes $F_t=cl\{f_t\}$ where all the assumptions of Theorem \ref{theo:submart_representation} are satisfied.
		\item[2.] The generalization of Theorem \ref{theo:zhang2020remarks} may be obtained for non-zero origin set-valued submartingales if the Hukuhara different (H-different) \cite{stefanini2019new}, $F_t{{\ominus }_{H}} F_0$, exists for all $t\ge 0$. Then, $F_t=F_0+(F_t\ominus_{H} F_0)$ and thus $F_t\ominus_{H} F_0$ is a zero origin set-valued submartingale. However, the existence of H-different does not always hold for every sets. Therefore, the integral representation for non-zero origin set-valued submartingales is still a open problem.
		\item[3.] The condition $E\left[ f_{t}^{k}|{\mathcal{F}}_{s} \right]\in S_{F_s}(\mathcal F_s)$ in Theorem \ref{theo:submart_representation} is not sufficient for the submartingale $F=(F_t)_{t\ge 0}$ to become a set-valued martingale. This means that the theorem is not in the generated case. To illustrate this remark, we give the following examples.
	\end{itemize}
\end{remark}
\begin{example}\label{examp:submart_repres}
	Given $u,v\in \mathcal{L}^{2}(\mathbb{R})$ are two different stochastic processes. Let $\displaystyle f_t=\int_{0}^{t}u_s d{B_s}$ and $\displaystyle g_t=\int_{0}^{t}{v_s}dB_s.$ For each $t\in [0,T]$, we define
	$$F_t=\text{cl}\left\{ \lambda {{f}_{t}}+\left( 1-\lambda  \right){{g}_{t}}:\lambda \in [0,1]\cap \mathbb{Q} \right\}.$$
	Then, $\{F_t\}, t\in [0,T]$ is a non-degenerated set-valued submartingale satisfying all conditions of Theorem \ref{theo:submart_representation} and thus its integral representation holds.
\end{example}
	
	Indeed, it is clear that $F_t=[\min\{f_t,g_t\},\max\{f_t,g_t\}]$. 
As both $\min\{f_t,g_t\}$ and $\max\{f_t,g_t\}$ are real-valued supermartingales, $F_t$ is a non-degenerate interval-valued submartingale. Since $F_0=\{0\}$ a.s., $\{F_t:t\in [0,T]\}$ is not an interval-valued martingale.
	
One can see that ${{F}_{t}}=\text{cl}\left\{ h_{t,t}^{\lambda }:\lambda \in [0,1]\cap \mathbb{Q} \right\}$ with
\[	
h_{t,t}^{\lambda }=h_{t}^{\lambda }=\lambda {{f}_{t}}+\left( 1-\lambda  \right){{g}_{t}}.
\]

Notice that
	  $\text{cl}\left\{ h_{t,t}^{\lambda }:\lambda \in [0,1]\cap \mathbb{Q} \right\}=\text{cl}\left\{ h_{r,t}^{\lambda }:\lambda \in [0,1]\cap \mathbb{Q},0\le r\le t \right\}$ since $h_{r,t}^{\lambda }\in {S}_{F_t}\left( {\mathcal{F}_{t}} \right)$ for all $0\le r\le t$, where $h_{r,t}^{\lambda}$ is designed as in Lemma \ref{lem:submartingale_Castaning}. From $h_{0,0}^{\lambda}=0$ a.s, we have $h_{r,0}^{\lambda}=0$ a.s for all $r\ge 0$.  On the other hand, for every $s<r$, 
	$E\left[ h_{r,r}^{\lambda }|{\mathcal{F}_{s}} \right]=E\left[ \lambda {{f}_{r}}+\left( 1-\lambda  \right){g_r}|{\mathcal{F}_{s}} \right]$ $=\lambda {f_s}+\left( 1-\lambda  \right){g_s}=h_{s,s}^{\lambda }$ $\in S_{F_s}\left( {\mathcal{F}_{s}} \right)$.
	Hence, $\{F_t:t\in [0,T]\}$ satisfies all the conditions of Theorem \ref{theo:submart_representation}. 
	
	Moreover, we have 
	\begin{align*}
		S_{F_t}(\mathcal F_t)&=\overline{\text{dec}}\left\{ h_{t}^{\lambda }:\lambda \in [0,1]\cap \mathbb{Q} \right\} \\ 
		& =\overline{\text{dec}}\left\{ \mathop{\int }_{0}^{t}\left( \lambda {u_s}+(1-\lambda ){v_s} \right)d{B_s}:\lambda \in [0,1]\cap \mathbb{Q} \right\} \\ 
		& =S_{\int_{0}^{t}Gd{B_s}}(\mathcal F_t), 
	\end{align*}
	where $$ G={{\left( {{G}_{t}} \right)}_{t\ge 0}}={{\left( \left\{ \lambda {{u}_{t}}+(1-\lambda ){{v}_{t}}:\lambda \in [0,1]\cap \mathbb{Q} \right\} \right)}_{t\ge 0}}.$$
	
	Thus, $\displaystyle {F}_{t}=\int_{0}^{t}Gd{B_s}$.

\begin{example}\label{examp:submart_repres_multi_demension}
	Let $\displaystyle \eta_t=\int_0^t u_s dB_s$ where $u\in \mathcal{L}^2(\mathbb R)$ is a stochastic process. We denote by $\mathbf B$ a countable set in $\XX=\mathbb R \times\mathbb R$ that $B=\{(i,j)\in\mathbb Q\times\mathbb Q:i^2+j^2\le 1\}$. Then, $\mathbf B$ is dense in $\XX$ and $cl(\mathbf B)$ is the unit ball in $\XX$. For each $t\ge 0$, let
	$$F_t=cl\{\eta_t \mathbf B\},$$
	where the closure is taken in $\mathbb R\times\mathbb R$. Then, $\{F_t,t\ge 0\}$ is also a non-degenerated set-valued submartingale satisfying all conditions of Theorem \ref{theo:submart_representation} and thus its integral representation holds.
\end{example}
	Indeed, $F_0=\{0\}$ a.s. is obviously. From $\mathbf B$ is countable, the Castaing representation of $F_t$ can be written by $F_t=cl\{f_t^k=\eta_t b_k:b_k\in B, k\in\mathbb N\}$. Hence, we have $E[f_t^k|\mathcal{F}_s]=E[\eta_t|\mathcal{F}_s]b_k =\eta_s b_k \in S_{F_s}(\mathcal{F}_s)$. Moreover, since $E[\eta_t|\mathcal{F}_s]b=E[\eta_tb|\mathcal{F}_s]$ for all $b\in \mathbf B$, we have $E[\eta_t|\mathcal{F}_s]\mathbf B\subset E[\eta_t \mathbf B|\mathcal{F}_s]$ and thus $E[F_t|\mathcal{F}_s]\supset F_s$. From this fact and the notice that $cl(\mathbf B)$ is the unit ball in $\XX$, $\{F_t,t\ge 0\}$ is a non-degenerated set-valued submartingale. Now, all the conditions of Theorem \ref{theo:submart_representation} are satisfied.
	
\subsection{Integral representation for non-trivial initial set-valued martingales}
The primary challenge in applying set-valued random variables to practical problems lies in the fact that a set-valued martingale with a trivial initial value (i.e., a singleton) will degenerate into a singleton. Therefore, our focus is exclusively on non-trivial initial set-valued martingales. 
However, the stochastic integral used in Theorems \ref{theo:kisielewicz2014},  \ref{theo:submart_representation} and \ref{theo:zhang2020remarks} unfortunately assumes zero initial value. 
Even that with certain assumptions, the set-valued martingale $(M_t)_{t\in [0,T]}$ can be represented in the form $\displaystyle M_t=E(M_0)+\int_0^t G_sdB_s$; and so  $(G_s)_{s\in [0,T]}$  will degenerate to singleton, which is still being a special case of set-valued martingales. 

To establish a stochastic integral representation for general initial set-valued martingales, it is necessary to extend the stochastic integral on the expanded space $\mathscr{L}_{\mathbb{F}}^2:=\XX\times\mathcal{L}_{\mathbb{F}}^2$. The added element $\XX$ in $\mathscr{L}_{\mathbb{F}}^2$ will store the initial values of set-valued martingales. 
In this section, we will work with $\XX=\R^d$ and so $\mathscr{L}_{\mathbb{F}}^2:=\R^d\times L^2([0,T]\times\Omega,\mathcal{F},\mathbb{F};\R^{d\times m})$.

Building on the work presented in  \cite{ararat2023set},  we now streamline the definition of the set-valued stochastic integral on $\mathscr{L}_{\mathbb F}^2$. Given an element $(x,\varphi)\in \mathscr{L}_{\mathbb{F}}^2$, for $t\in [0,T]$, define the mapping $\mathcal{J}_B^t$ by setting $\displaystyle \mathcal{J}_B^t(x,\varphi)=x+\int_0^t\varphi_\tau dB_\tau$. It is clear that $(\mathcal{J}_B^t(x,\varphi))_{t\in [0,T]}$ is a continuous $\mathbb{F}$-martingale. Let $\mathcal{G}\in \mathcal{P}(\mathscr{L}_{\mathbb{F}}^2)$ (the family of all nonempty subsets of $\mathscr{L}_{\mathbb{F}}^2$). For each $t\in [0,T]$, there exists only one $\mathcal{F}_t$-measurable set-valued random variable $F_t:\Omega \to \mathbf{K}(\R^d)$ such that $S_{F_t}(\mathcal{F}_t)=\overline{\text{dec}}\mathcal{J}_B^t(\mathcal G)$ (see Theorem \ref{theo:closed_set_selection}).

\begin{definition}
	Let $\mathcal{G}\in \mathcal{P}(\mathscr{L}_{\mathbb{F}}^2)$.  For $t\in [0,T]$, there exists only one $\mathcal{F}_t$-measurable set-valued random variable on $L^2[\Omega,\mathcal{F}_t;\mathbf{K}(\R^d)]$, denoted by $\displaystyle \int_{\Theta}^t\mathcal{G}dBs$, such that $S_{\int_{\Theta}^t\mathcal{G}dBs}(\mathcal{F}_t)=\overline{\text{dec}}\mathcal{J}_B^t(\mathcal G)$. The set-valued random variable $\displaystyle\int_{\Theta}^t\mathcal{G}dBs$ is called the generalized set-valued stochastic integral of $\mathcal{G}$ on $\mathscr{L}_{\mathbb{F}}^2$.
\end{definition}

\begin{remark}\label{re:3.2}\ 
	\begin{itemize}
		\item[1.] If $\mathcal{G}=\{x\}\times G$ for some $x\in\R^d$ and $G\subseteq\mathcal{L}_{\mathbb{F}}^2$, then $\displaystyle\int_{\Theta}^t\mathcal{G}dBs=\{x\}+\int_0^t GdB_s$. Hence, when $x=0$ we have $\displaystyle\int_{\Theta}^t\mathcal{G}dBs=\int_0^t GdB_s$.
		\item[2.] Since $\mathcal{J}_B^t(\mathcal{G})$ is a family of $\mathbb{F}$-martingale, $\displaystyle\int_{\Theta}^t\mathcal{G}dBs$ is generally a set-valued submartingale. Indeed, we have
		\begin{equation}\label{eq:3.3}
			S^2_{E[\int_{\Theta}^t\mathcal{G}dB|\mathcal{F}_s]}(\mathcal{F}_s)=\cl\left\{E[\psi|\mathcal{F}_s]:\psi\in\overline{\text{dec}}\mathcal{J}_B^t(\mathcal G)\right\}.
		\end{equation}
		We will show that $\left\{E[\psi|\mathcal{F}_s]:\psi\in\overline{\text{dec}}\mathcal{J}_B^t(\mathcal G)\right\}$ is $\mathcal{F}_s$-decomposable. Let $\xi_1,\xi_2\in \left\{E[\psi|\mathcal{F}_s]:\psi\in\overline{\text{dec}}\mathcal{J}_B^t(\mathcal G)\right\}$. For any $A\in\mathcal{F}_s$, there exists $\psi_1,\psi_2\in \overline{\text{dec}}\mathcal{J}_B^t(\mathcal G)$ such that $\xi_1=E[\psi_1|\mathcal{F}_s]$ and $\xi_2=E[\psi_2|\mathcal{F}_s]$. We have 
		\begin{align}\label{eq:3.4}
			I_A\xi_1+I_{\Omega\setminus A}\xi_2&=I_AE[\psi_1|\mathcal{F}_s]+I_{\Omega\setminus A}E[\psi_2|\mathcal{F}_s]\notag\\ 
			&=E[I_A\psi_1|\mathcal{F}_s]+E[I_{\Omega\setminus A}\psi_2|\mathcal{F}_s]\notag\\ 
			&=E[I_A\psi_1+I_{\Omega\setminus A}\psi_2|\mathcal{F}_s]\notag\\ 
			&\in \left\{E[\psi|\mathcal{F}_s]:\psi\in\overline{\text{dec}}\mathcal{J}_B^t(\mathcal G)\right\}.
		\end{align} 
	The relation ``$\in$'' in equation \eqref{eq:3.4} is due to $I_A\psi_1+I_{\Omega\setminus A}\psi_2 \in \overline{\text{dec}}\mathcal{J}_B^t(\mathcal G)$, $A\in\mathcal{F}_s\subset\mathcal{F}_t$, and the decomposability of $\overline{\text{dec}}\mathcal{J}_B^t(\mathcal G)$.\\
	 Thus, $\left\{E[\psi|\mathcal{F}_s]:\psi\in\overline{\text{dec}}\mathcal{J}_B^t(\mathcal G)\right\}$ is $\mathcal{F}_s$-decomposable. On the other hand, for all $(x,\varphi)\in \mathcal{G}$, we have $\mathcal{J}_B^s(x,\varphi)=E[\mathcal{J}_B^t(x,\varphi)|\mathcal{F}_s]$. Hence,
	 $$ \mathcal{J}_B^s(\mathcal G)\subseteq \left\{E[\psi|\mathcal{F}_s]:\psi\in\overline{\text{dec}}\mathcal{J}_B^t(\mathcal G)\right\}.$$
	 From the definition of decomposable hull of a set, we have
	 $$\text{dec}\mathcal{J}_B^s(\mathcal G)\subseteq \left\{E[\psi|\mathcal{F}_s]:\psi\in\overline{\text{dec}}\mathcal{J}_B^t(\mathcal G)\right\},$$
	 implying that
	 \begin{equation}\label{eq:3.5}
	 	\overline{\text{dec}}\mathcal{J}_B^s(\mathcal G)\subseteq \cl\left\{E[\psi|\mathcal{F}_s]:\psi\in\overline{\text{dec}}\mathcal{J}_B^t(\mathcal G)\right\}.
	 \end{equation}
	In view of \eqref{eq:3.3} and \eqref{eq:3.5}, we have
	 $$
	 	S^2_{E[\int_{\Theta}^t\mathcal{G}dB|\mathcal{F}_s]}(\mathcal{F}_s) \supseteq\overline{\text{dec}}\mathcal{J}_B^s(\mathcal G)
	 	=  S^2_{\int_{\Theta}^s\mathcal{G}dB}(\mathcal{F}_s).
	$$
 	Therefore, $\displaystyle E\left[\int_{\Theta}^t\mathcal{G}dB|\mathcal{F}_s\right]\supseteq \int_{\Theta}^s\mathcal{G}dB$. This implies that $\displaystyle\int_{\Theta}^t\mathcal{G}dBs$ is a set-valued submartingale.
	\end{itemize}
\end{remark}

Let $M=(M_t)_{t\in [0,T]}$ be a convex uniformly square integrably bounded set-valued $\mathbb{F}$-martingale. We denote by $MS(M)$ the set of all square integrably bounded $\mathbb{F}$-martingale selectors of $M$ and $P_t[MS(M)]$ the $t$th-sector of $MS(M)$. As a result of traditional martingale representation, for every martingale selector $f=(f_t)_{t\in [0,T]}\in MS(M)$, there exists only one $(x^f,\varphi^f)\in \mathscr{F}_{\mathbb F}^2$ such that $\displaystyle f_t=x^f+\int_0^t \varphi_{\tau}^fdB_{\tau}$ a.s. for every $t\in [0,T]$. Let us denote it by $\mathcal{G}^M=\left\{(x^f,\varphi^f)\in\mathscr{F}_{\mathbb F}^2: f\in MS(M)\right\}$.

\begin{theorem}
	Let $M=(M_t)_{t\in [0,T]}$ be a convex uniformly square integrably bounded set-valued $\mathbb{F}$-martingale. Then, for every $t\in [0,T]$, 
	$$ M_t=\int_{\Theta}^t \mathcal{G}^MdB_{\tau} \ \,a.s.$$
\end{theorem}
\begin{proof} 
	For each $t\in [0,T]$, by the definition of $\int_{\Theta}^t \mathcal{G}^MdB_{\tau}$ we have $$S_{\int_{\Theta}^t\mathcal{G}^MdB}^2(\mathcal{F}_t)=\overline{\text{dec}}\mathcal{J}_B^t(\mathcal {G}^M).$$
	It is clear that $\mathcal{J}_B^t(\mathcal {G}^M)=P_t[MS(M)]$. Hence, 
	$$S_{\int_{\Theta}^t\mathcal{G}^MdB}^2(\mathcal{F}_t)=\overline{\text{dec}}P_t[MS(M)].$$
	Using the result of [Proposition 3.1,\cite{kisielewicz2014martingale}], we get $\overline{\text{dec}}P_t[MS(M)]=S_{M_t}^2(\mathcal{F}_t)$. 
	Thus,
	$$S_{\int_{\Theta}^t\mathcal{G}^MdB}^2(\mathcal{F}_t)=S_{M_t}^2(\mathcal{F}_t),$$
	showing that $\displaystyle M_t =\int_{\Theta}^t\mathcal{G}^MdB$ a.s.
\end{proof}

\begin{remark}\ 
	When $M_0=\{0\}$,  $\mathcal{G}^M=\{(0;\varphi^f):f\in MS(M)\}$. Therefore,\\ $\displaystyle \int_{\Theta}^t\mathcal{G}^MdB=\int_0^tG^MdB$, where $G^M=\left\{\varphi^f:f\in MS(M), f_t=\int_0^t\varphi^f_sdB_s, t\in [0,T]\right\}$. This results in the representation of \cite{kisielewicz2014martingale}.
\end{remark}

According to Remark \ref{re:3.2}, the set-valued stochastic integral $\displaystyle \int_{\Theta}^t \mathcal G dB_{\tau}$ is generally a set-valued submartingale. 
Inversely, it is not always true.
By using \ref{theo:submart_representation} and Lemma \ref{lem:submartingale_Castaning}, we can now give a stochastic integral representation for non-singleton initial set-valued submartingales.

\begin{corallary}
	Given $F=(F_t)_{t\ge 0}$ be a set-valued submartingale satisfying all the assumptions of Theorem \ref{theo:submart_representation} with the exception that $F_0$ is non-singleton. Then, there exists a set $\mathcal{G}\in\mathscr{F}_{\mathbb F}^2$ such that 
 \begin{equation}\label{eq:3.6}	
 F_t=\int_{\Theta}^t \mathcal G dB_{\tau}\\, \, a.s. 
\end{equation}	
\end{corallary}

\begin{proof}
	The proof can be done similarly to Theorem \ref{theo:submart_representation}. By changing $\displaystyle f_{r,t}^k=x^{f_r^k}+\int_0^t g_{\tau}^{f_r^k}dB_{\tau}$ and let $\mathcal{G}=\left\{(x^{f_r^k},g^{f_r^k}):f_r^k \text{ is introduced in Lemma \ref{lem:submartingale_Castaning}}\right\}$, one obtains the representation \eqref{eq:3.6}.
\end{proof}

\section{Conclusions}
In this work, we have shown that every set-valued martingale with zero origin degenerates to the singleton case.
 To avoid the degenerated case of set-valued martingales, we have introduced a general integral representation theorem for non-degenerated set-valued submartingales based on certain assumptions. Finally, we also present an extended stochastic integral representation for non-singleton initial set-valued martingales.
Using this work allows one to recover the previous results.

\section*{Acknowledgements}
The authors would like to thank the two anonymous referees for their valuable comments and useful suggestions that helped to improve the quality of the paper.

\section*{Funding}
The first author, L.T. Tuyen, was supported by  Institute  of  Information  Technology, 
Vietnam  Academy  of  Scinece  and  Technology, under the project no. CS.23.02.

\section*{ORCID}
Luc Tri Tuyen: 0000-0002-4822-2978

Vu Thai Luan: 0000-0003-0319-654X
\section*{Conflict of interest}
The authors declare that they have no conflict of interest.

\bibliographystyle{plain}      
\bibliography{References}   

\end{document}